\documentclass[11pt,leqno]{article}
\usepackage{amsthm,amsfonts,amssymb,epsfig,graphics,amsmath}
\usepackage [round]{natbib}
\usepackage[latin1]{inputenc}\relax
\usepackage{hyperref}
\hypersetup{
    colorlinks,
    citecolor=blue,
    filecolor=black,
    linkcolor=black,
    urlcolor=blue,
   }
\newtheorem{theo}{Theorem}[section]

\newtheorem{cor}[theo]{Corollary}
\newtheorem{lem}[theo]{Lemma}
\newtheorem{defi}[theo]{Definition}

\newtheorem{rem}[theo]{Remark}

\numberwithin{equation}{section}
\newcommand\be{\begin{equation}}
\newcommand\ee{\end{equation}}
\newcommand\bp{\begin{pmatrix}}
\newcommand\ep{\end{pmatrix}}
\newcommand\ba{\begin{aligned}}
\newcommand\ea{\end{aligned} }

\newcommand\bea{\begin{array}}
\newcommand\ena{\end{array}}

\begin{document}

\title{\bf On classes of life distributions: Dichotomous Markov Noise Shock Model With Hypothesis Testing Applications}
\author{\sc \small
Mohammad Sepehrifar\thanks{Mississippi State University, MS 39762;
Department of Mathematics and Statistics, msepehrifar@math.msstate.edu.},
Shantia Yarahmadian\thanks{ Mississippi State University, MS 39762;
Department of Mathematics and Statistics, syarahmadian@math.msstate.edu.},
Richard Yamada\thanks{Duke University, Institute for Genome Sciences and Policy; yry@duke.edu.}
}

\maketitle

\begin{abstract}
In this paper, we investigate the probabilistic characteristics of a unit driven by Dichotomous Markov Noise (DMN), as an external random life increasing and decreasing shocks. Using DMN, we will define two new aging classes of the overall increasing/decreasing (OIL/ODL) nature in the long time behavior, which are separated by an exponential steady state regime. In addition, a moment inequality is derived for the system whose life distribution is in an overall life decreasing (ODL) class. We use this inequality to devise a nonparametric testing procedure for exponentiality against an alternative overall decreasing life distribution. 
\end{abstract}

\paragraph{General Theory}
The statistical theory of positive aging systems is an interdisciplinary domain of research, which is encountered in many physical, chemical, biological, and engineering safety systems. Here, we study the positive aging process driven by a Dichotomous Markov Noise (DMN). Describing the deterioration or positive aging of life in the engineering and biological systems, based on different aging criteria, has been the subject of investigation for several decades. For more technical details we refer the reader to \citet {MO}, \citet{SHSH}, \citet{R1}, \citet{BP}, \citet{DKS1}, \citet{A1}, \citet{MA}, and \citet{IS}. The plan of this paper is as follows: In section one, we introduce the concept of Dichotomous Markov Noise (DMN) and  introduce two new aging classes of  Overall Decreasing Life (ODL) and Overall Increasing Life (OIL). The moment inequality approach is introduced in section two. In section three, a new statistics test, based on U-statistics theory is established to test exponentiality against an alternative distribution $F$, where $F$ belongs to an overall decreasing life (ODL) class of distributions. In section four, Monte Carlo simulations for sample sizes $n=10, 20, 30$ are conducted for this test. The calculated empirical power of the test shows the excellent power of this test for some common alternative distributions. The novelty in this study is the development of a nonparametric testing procedure for testing expoenetiality against the two new aging classes for a unit driven by Dichotomous Markov Noise (DMN). Furthermore, this proposed test is of special interest to biostatisticians, who are typically contend and are limited to data that comes at successive patient visits to a facility; such data ignores crucial events that may occur between visits. 

\section{Introduction}\label{intro}
Let $X(t)$, a continuous positive random variable, describing the lifetime (aging) of a mechanical unit, living organism or some other physical, electrical or biological system \citep{DKS1}. We interpret aging as the time that the system has worked satisfactorily without experiencing a failure and then enters into a failed state \citep{DKS1, KW}. In this context, a system with the older age has a shorter lifetime \citep{KW}. In addition, we assume that $X(t)$ is described by Dichotomous Markov  stochastic process. 
    
\subsection{Dichotomous Markov Noise (DMN)}
First order systems driven by dichotomous noise are of practical interest in a wide variety of biological, physical and engineering problems \citep{B1, B2, YBZS}. Similar to an Ornstein-Uhlenbeck process, DMN is colored noise and defined as a two-valued stochastic process with state space values $v_{\pm}$ and constant transition frequencies of $\lambda_{\pm}$; the increase $(+)$ and decrease rate $(-)$ for $\lambda_{\pm}$ occur with probabilities $p_{\pm}(t)$. The switching of $v(t)$ are Poisson processes. The evolution of the unit age, $X(t)$, is represented by the following system of differential equations \citep{B1, YBZS}:

\begin{figure}[h!]
\centering
\includegraphics[scale=.7]{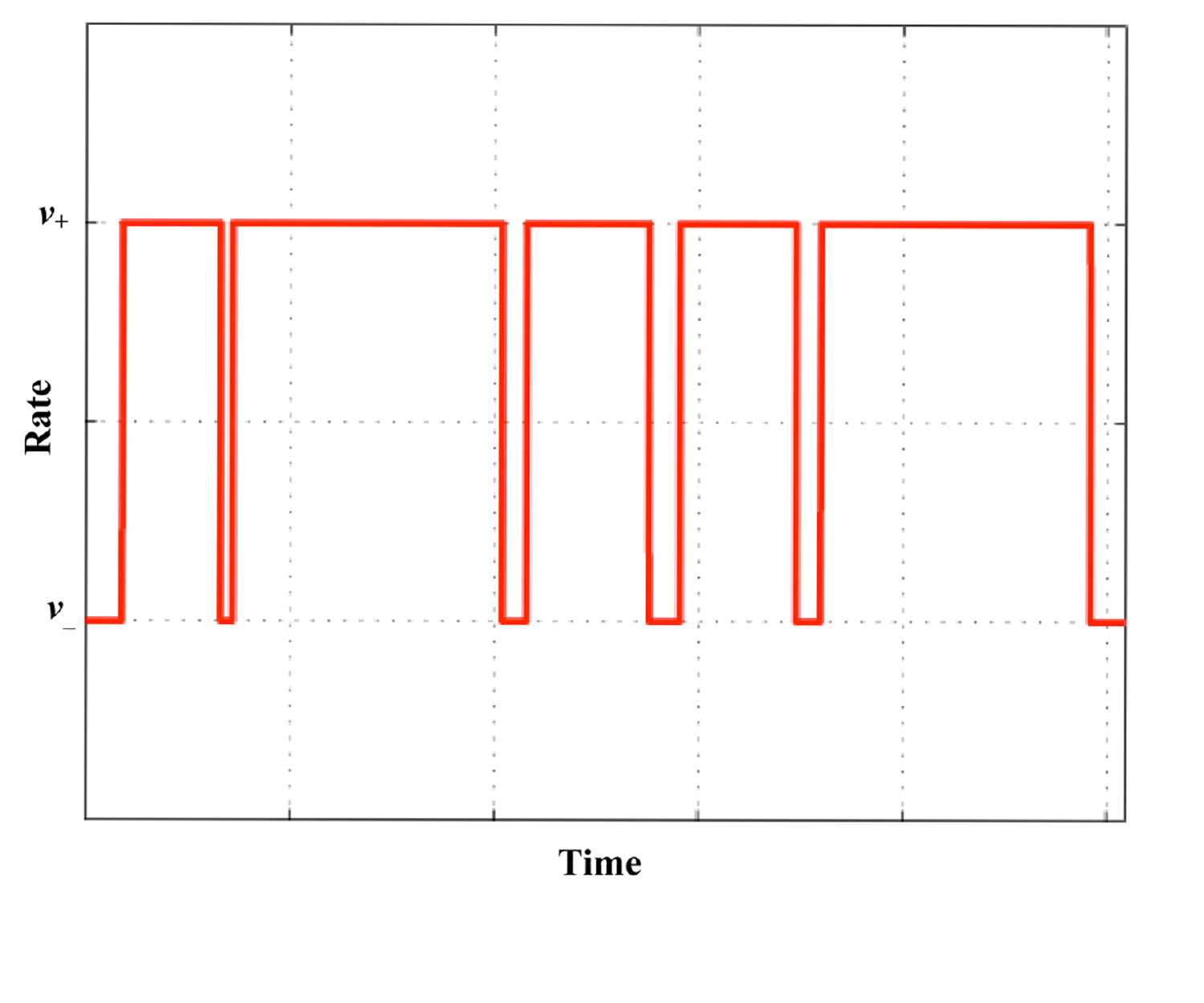}
\caption{Dichotomous Markov Noise.}
\end{figure}

\be
\frac{d}{dt}\bp{p}_+ \\ {p}_- \ep=\bp -\lambda_+ && \lambda_- \\  \lambda_+  && -\lambda_- \ep
\bp {p}_+ \\ {p}_- \ep
\ee

Using the initial conditions $p_+(0)=1$ and $p_-(0)=0$, we can represent the solutions as follows:

\be \label{p1}
{p}_+(t) =\frac{\lambda_-}{{\lambda_++\lambda_-}}+ \frac{\lambda_+}{{\lambda_++\lambda_-}}e^{-{(\lambda_++\lambda_-)}t}
\ee
\be \label{p2}
{p}_-(t) =\frac{\lambda_+}{{\lambda_++\lambda_-}}-\frac{\lambda_+}{{\lambda_++\lambda_-}}e^{-{(\lambda_++\lambda_-)}t}
\ee

where ${p}_+(t)+{p}_-(t)=1$ for all times, and $\frac{1}{\lambda}=\frac{1}{{\lambda_++\lambda_-}}$ represents the mean time between switches of $v(t)$, otherwise known as the rate relaxation time.

\begin{rem}
DMN is a time-homogeneous Markov process and is completely characterized by the following transition probabilities:

\be
P_{ij}(t)=Pr(v(t)=v_i|v(0)=v_j), \qquad i,j \in \{{\pm}\} 
\ee

The temporal evolution of transition probabilities are written as:

\begin{equation}
\frac{d}{dt}\bp P_{-j}(t)\\ P_{+j}(t)\ep=\bp -\lambda_+ && \lambda_- \\  \lambda_+ && -\lambda_- \ep\bp P_{-j}(t)\\ P_{+j}(t)\ep
\end{equation}

with the following solutions:

\begin{equation}
\bp
P_{--}(t)&P_{-+}(t)\\
P_{+-}(t)&P_{++}(t)
\ep=\frac{1}{\lambda_++\lambda_-} \bp \lambda_++\lambda_-e^{-(\lambda_++\lambda_-)t}&&\lambda_+(1-e^{-(\lambda_++\lambda_-)t})\\
\lambda_-(1-e^{-(\lambda_++\lambda_-)t})&&\lambda_-+\lambda_+e^{-(\lambda_++\lambda_-)t} \ep
\end{equation}

In the stationary case we have:
\be \label{stationary}
Pr(v=v_+)=\frac{\lambda_-}{\lambda_++\lambda_-}, \quad Pr(v=v_-)=\frac{\lambda_+}{\lambda_++\lambda_-}
\ee
 \end{rem}

\begin{lem}
The aging life of the system is characterized by $V=\frac{v_+\lambda_--v_-\lambda_+}{\lambda_++\lambda_-}$.
\end{lem}

\begin{proof}
We define the instantaneous average rate  as:
\be
\overline{v}(t)=v_+{p_+}(t)-v_-{p_-(t)}
\ee
Using \eqref{p1} and \eqref{p1}, we get
\be
\overline{v}(t)=V(1-e^{-(\lambda_++\lambda_-)t})+v_+e^{-(\lambda_++\lambda_-)t}
\ee
and
\be
V=\lim_{t \to \infty}\overline{v}(t)=\frac{v_+\lambda_--v_-\lambda_+}{\lambda_++\lambda_-}
\ee
\end{proof}

\begin{defi}
When $V<0$, life decreasing shocks are more frequent and the system is in an Overall Decreasing Life (ODL) class, whereas when $V>0$ life increasing shocks are more frequent and the system is in an Overall Increasing Life (OIL) class.\\
\end{defi}

The following lemma captures the situation when the life distribution of the system falls in the steady-state:
\begin{lem}
In steady-state, the mean unit age and its probability distribution function (pdf), are obtained through an exponential distribution:
\be \label{valid}
\Lambda=\frac{v_+  v_-}{v_- \lambda_+-v_+  \lambda_- }~~~~~f(x)=\frac{1}{\Lambda} e^{-\frac{x}{\Lambda}}
\ee
\end{lem}
\begin{proof}

The time evolution of the unit age distribution $P_{\pm} (x,t)$ is represented by:
\be \label{ev}
\frac{d}{dt}\bp P_+(x,t)\\P_-(x,t)\ep=
\bp
-v_+\frac{\partial}{\partial x}-\lambda_+ && \lambda_-\\ \lambda_+&&v_-\frac{\partial}{\partial x}-\lambda_-
\ep
\bp P_+(x,t)\\P_-(x,t)\ep
\ee

In the steady state, since there is no dependence on time , the time derivatives are set to zero and \eqref{ev} is reduced to the following:

\be \label{ev1}
\frac{d}{dx}\bp P_+(x,t)\\P_-(x,t)\ep=
\bp
-\frac{\lambda_+}{v_+} && \frac{\lambda_-}{v_+}\\ -\frac{\lambda_+}{v_-} &&\frac{\lambda_-}{v_-} 
\ep
\bp P_+(x,t)\\P_-(x,t)\ep
\ee

or equivalently to:

\be
\frac{d^2P_{\pm}}{dx^2}+\frac{1}{\Lambda}\frac{dP_{\pm}}{dt}=0
\ee

Using the fact that $\lim_{x \to \infty} P_{\pm}(x)=0$, the solutions to \eqref{ev1} are written as:
\be
P_{\pm}(x)=A_{\pm}e^{-\frac{x}{\Lambda}}
\ee
where $A_{\pm}$ are two real constants. Now  $f(x)=P_+(x)+P_-(x)$ and  $\int_0^{\infty}f(x)dx=1$ will establish the result.
\end{proof}

\section{Statistical Approach:  Moment Inequality Approach}
Let $F(x)$ be the c.d.f. of $X(t)$ and let ${\bar F}(x)=1-F(x)=Pr[X(t)> x]$ be the survival function. The survival function of a unit of age $x_0$, $F_{x_0}(x)$, is defined as the conditional probability that a unit of age $x_0$ (i.e., a unit having the initial age $x_0$) will survive for an additional $x$ units of time:

\be
{\bar F}(x|x_0)=Pr[X (t)> x + x_0  | X(t) > x_0 ]  =\frac {{\bar F}(x + x_0 )}{{\bar F}(x_0)}
\ee

If $x_0=0$, then by definition ${\bar F}(x|0)=\bar F(x)$ is the survival function of a new unit. Any study concerning the phenomenon of aging must be based on ${\bar F}(x|x_0)$ and functions related to it. Associated with $X$ is the notion of " random remaining life" at age $x_0$, denoted by $X_{x_0}$ which has a survival function ${\frac{{\bar F}(x + x_0 )}{{\bar F}(x_0)}}$.The concept of mean remaining life (MRL) has found applications in biometry, actuarial science, and engineering, nonparametric life testing and reliability analysis. For reviews of the applications of MRL, the reader can refer to Muth \citep{M1}, Hall and Wellner \citep{HW}, Gupta and Gupta \citep{GG} and the various references cited in these works.
\\
Now, let $X_1, X_2, ...$ be a sequence of independent identically distributed random lives with survival function ${\bar F}=1-F$ and finite mean $\Lambda$ ( $\Lambda$ is described by \eqref{valid}).
\begin{rem}
The distribution \eqref{valid} is meaningful if and only if  $\frac{f(x)}{\bar F(x)}$ $=$ $\frac{1}{\Lambda}$ .
\end{rem}
\begin{rem}
 Abouammoh et. al. \citep{AAB} introduced a model of accumulated additive damages that cause the failure of a device when it exceeds a certain value threshold. Abouammoh et. al. \citep{AAQ} also established an empirical test statistics for testing exponentiality when the alternative to exponentiality belongs to the new better than renewal used in expectation (NBRUE). 
 \end{rem}

 Let $N(t)$ be the total number of shocks governed by Poisson process, which take place in $[0,t)$. Then the remaining life at $t$ is $L(t)=\sum_{i=1}^{N(t)+1}X_i
-t$, assuming $\sum_{i=1}^{N(t)}X_i-x>0$ where $x$ is the fixed threshold. In this setting $L(t)$ converges weekly to a renewal random life denoted by ${\tilde X}$ with survival function ${\tilde W}(x)=\frac{1}{\Lambda}\int_x^{\infty} {\bar F}(u) du$ \citep{IM}. Note that ${\tilde X}$ can be studied through the remaining (residual)
life $X_{x_0}=[X-x_0 | X>x_0]$ at age $x_0$.

\begin {theo} \label{MRL}
For all $x \ge 0$, ${\tilde W}(x)=\frac{1}{\Lambda} E(X-x)I(X-x)$  where $I(u)=1$, if $u > 0$ and $0$ otherwise.
\end{theo}
\begin{proof}
see \citep{IM}.
\end{proof}

\begin{defi}\label{O}
A life distribution $F$ on $(0, \infty)$, with $F(0^-)=0$  is called overall decreasing life (ODL), if 
\be \label{overall}
\int_t^{\infty}{\tilde W}(x)dx \le \Lambda {\tilde W}(t), ~~~~t \ge 0, 
\ee
where $\Lambda=\int_0^{\infty}{\bar F}(u)du < \infty$.
\end{defi}

\begin{rem}
Life distribution F belongs to $ODL$ if and only if $\frac{f(x)}{\bar F(x)}$ $<$ $\frac{1}{\Lambda}$.
\end{rem}

There are a number of classes that have been suggested in the literature to categorize distributions based on their aging properties or their dual. Among the most practical aging classes of life distributions, we study the {\it Increasing Failure Rate (IFR)} . A survival distribution is said to be in the $IFR$ class if and only if ${\frac{{\bar F}(x + x_0 )}{{\bar F}(x_0)}}$ is decreasing  in $x$ for all $x_0$. The definition says that the probability of an individual of age $x$ survives an additional $x_0$ period of time is decreasing with time.
 
\begin{rem} \label{IFR}
Barlow and Proschan \citep{BP} showed that an $IFR$ life distribution $F$ with mean $\Lambda$ is bounded below by $e^{-\frac{t}{\Lambda}} ~~ \text{for}~ t \le \Lambda$.
\end{rem}

\begin{lem}
The implications between the $IFR$ and $ODL$ classes are as follows:
$$IFR \Rightarrow ODL $$
\end{lem}
\begin{proof}
$IFR \Rightarrow ODL$:\\
This follows from direct application of Remark 3.4 and Remark 3.5.\

\end{proof}

\begin{theo}
If a lifetime distribution $F$ belongs to the ODL class, then for all integers $r \ge 0$
\be \label{ineq}
E\Big\{\frac{1}{2(r+1)}X_1^{r+1}X_2^{2}-\frac{1}{r+2}X_1^{r+2}X_2+\frac{1}{2(r+3)}X_1^{r+3}\Big\} \le \Lambda E\Big\{\frac{1}{r+1}X_1M_{1,2}^{r+1}-\frac{1}{r+2}X_1M_{1,2}^{r+2}\Big\}
\ee 
where $X_1,~~X_2$ are two stochastically independent copies of life $X$ and $M_{1,2}=min(X_1,X_2)$.
\end{theo}

\begin{proof}
$F$ belongs to the ODL class  if $\int_t^{\infty}{\tilde W}(x)dx \le \Lambda {\tilde W}(t)$ . Apply Theorem \ref{MRL}, one may rewrite \eqref{overall} as follows: 
\be
\int_0^{\infty} \int_t^{\infty} t^r {\bar F}(t){\tilde W}(x)dx dt \le \Lambda \int_0^{\infty} t^r {\bar F}(t){\tilde W}(t) dt
\ee
\be
\begin{split}
L.H.S &= \int_0^{\infty} \int_t^{\infty} t^r {\bar F}(t){\tilde W}(x)dx dt
\\
&= \frac{1}{\Lambda} E \Big \{\int_0^{X_1}\Big(\frac{1}{2}t^rX_2^2-t^{r+1}X_2+\frac{1}{2}t^{2+r}\Big )dt \Big\}
\\
&= E\Big\{\frac{1}{2(r+1)}X_1^{r+1}X_2^{2}-\frac{1}{r+2}X_1^{r+2}X_2+\frac{1}{2(r+3)}X_1^{r+3}\Big\}
\end{split}
\ee
and
\be
\begin{split}
R.H.S &=\Lambda \int_0^{\infty} t^r {\bar F}(t){\tilde W}(t) dt
\\
&=E\Big\{\int_0^{\infty}(X_1-t)I(X_1>t)t^r\bar{F}(t)dt\Big\}
\\
&=E\Big\{\int_0^{\min(X_1,X_2)}(X_1-t)I(X_1>t)t^r\bar{F}(t)dt\Big\}
\\
&= E\Big\{\frac{1}{r+1}X_1M_{1,2}^{r+1}-\frac{1}{r+2}X_1M_{1,2}^{r+2}\Big\}
\end{split}
\ee
\end{proof}

When $r=0$ we deduce the following special case:
\begin {cor} \label{13}
Setting $r=0$ in \eqref{ineq}:
\be
E\Big\{\frac{1}{2}X_1X_2^2-\frac{1}{2}X_1^2X_2+\frac{1}{6}X_1^3\Big\} \le \Lambda E\{X_1M_{1,2}-\frac{1}{2}M_{1,2}^2\}
\ee
\end{cor}

\section{Statistical Testing of  ${\tilde X}$ Aging (ODL Alternative) }
The goal of this section is to create a symmetric kernel, allowing the application of 
U-statistics theory for testing exponentiality against any other alternatives in the ODL aging class.  We use corollary \eqref{13} to test $H_0: F~ \text{is exponential}$ vs $H_a: F~ \text{is ODL}$.  We define the measure of departure from the null hypothesis as:

\be
\delta=E \Big\{\frac{1}{2}X_1X_2^2-\frac{1}{2}X_1^2X_2+\frac{1}{6}X_1^3-\Lambda(X_1M_{1,2}-\frac{1}{2}M^2_{1,2})\Big\}.
\ee
To make the test statistic scale invariate, we adjust $\delta$ by $\Lambda^3$ and define

\be
\Delta=\frac{1}{\Lambda^3}E \Big\{\frac{1}{2}X_1X_2^2-\frac{1}{2}X_1^2X_2+\frac{1}{6}X_1^3-\Lambda(X_1M_{1,2}-\frac{1}{2}M^2_{1,2})\Big\}
\ee
Note that $\Delta=0$ under the null hypothesis and $\Delta<0$ under the alternative.
\\\\
Let $X_1, X_2, \dots X_n$  denote a random sample from $F$. A U-statistics form of $\delta$ is defined as

\be
{\hat \delta}=\frac{2}{n(n-1)}\sum_{1\le i\le j \le n} \Big\{ \frac{1}{2}X_iX_j^2-\frac{1}{2}X_i^2X_j+\frac{1}{6}X_i^3-X_i^2M_{i,j}+\frac{1}{2}X_iM_{i,j}^2 \Big \}
\ee
Now, we can estimate $\Delta$ by ${\hat \Delta}=\frac{{\hat \delta}}{{\bar X}^3}$, where ${\bar X}=\frac{1}{n}\sum_i X_i$. To perform this test, we calculate ${\hat \Delta}$ from data and reject $H_0$ in favor of $H_1$ if the normal variate value $Z_{1-\alpha}$ exceeds ${\hat \sigma}_0^{-1}n^{\frac{1}{2}}{\hat \Delta}$. In this calculation one may find the value of variance through the standard argument of $U$-statistics theory \citep{L} as
follows:

\begin{theo}
As $n \to \infty$, $\sqrt{n}(\hat \Delta-\Delta_0)$ is asymptotically distributed normal with mean $0$ and variance

\be
\ba
\sigma^2=\frac{1}{\mu^6}Var \Big \{\frac{1}{6}X_1^3+\frac{1}{6}\int_0^{\infty}x^3dF(x)-\frac{1}{2}X_1^3{\bar F}(X_1)+\frac{3}{2}(X_1\int_0^{X_1}x^2dF(x)\\+\int_0^{X_1}x^3dF(x)-X_1^2\int_0^{X_1}xdF(x))\Big \}
\ea
\ee
under the null hypothesis ${\hat \sigma}_0=1.173$.
\end{theo}
\begin{proof}
Setting 
\be
\Phi(X_1,X_2)= \Big\{ \frac{1}{2}X_1X_2^2-\frac{1}{2}X_1^2X_2+\frac{1}{6}X_1^3-X_1^2M_{1,2}+\frac{1}{2}X_1M_{1,2}^2\Big\}
\ee
and then applying the standard theorem of U-statistics \citep{L}, we obtain
\be
\sigma^2=Var \Big\{E\{\Phi(X_1,X_2|X_1)\}+E\{\Phi(X_2,X_1|X_1)\}\Big\},
\ee
in which
\be
\ba
E\Big\{X_1^2\min(X_1,X_2)|X_1\Big\} &=E\Big\{X_1^2 \min(X_1,X_2)I(X_1<X_2)|X_1\Big\} + E\Big\{X_1^2 X_2I(X_1>X_2)|X_1\Big\}\\
&=X_1^3\bar{F}(X_1)+X_1^2\int_0^{X_1}xdF(x)
\ea
\ee
and
\be
\ba
E\Big\{X_2^2\min(X_2,X_1)|X_1\Big\} &= E\Big\{X_2^2 \min(X_2,X_1)I(X_1<X_2)|X_1\Big\} + E\Big\{X_2^2 \min(X_2,X_1)I(X_1>X_2)|X_1\Big\}\\
&=X_1\int_{X_1}^{\infty}(x^2+x^3)dF(x)
\ea
\ee
Direct calculations will establish the results.
\end{proof}

\section{Simulated Power Studies}
The empirical power estimates of our test statistics are obtained based on $10000$ replicates of the Monte Carlo method simulating common alternative life distributions with increasing failure rate. Weibull, Gamma, and linear failure rate family with the distribution functions,
\be 
{\bar F}_1(x)=e^{-x^{\theta}},~x>0,~~\theta \ge 1,
\ee 
\be 
{\bar F}_2(x)=e^{-x-\frac{\theta}{2}x^2},~x>0,~~\theta \ge 0,
\ee
and  
\be 
{\bar F}_3(x)=\frac{1}{\Gamma(\theta)}\int_x^{\infty}e^{-u}u^{\theta-1}du,~x>0,~~\theta \ge 0
\ee 
respectively, are among these alternatives.
\\
We report the power values for samples of size $n=10$, $n=20$, and $n=30$ in Table 1 when $1 \le \theta $. These estimates of power were obtained for tests where $\alpha=0.05$. More extensive simulations and results, along with the tables, are available upon request from the authors. We note from Table 1 that is $\theta$ $>$1 (increasing the distance of the exponential), the proposed test get better power; on the other hand the power decreases as $\theta$ approaches to $1$ , which reduces to the ordinary exponential distribution.  
\subsection{Numerical Example}

\begin{table}
 \begin{center}
\caption{\bf Power Estimates Using $\alpha=0.05$}
\begin{tabular}{|l|llll|}
 \hline
 \small{\bf Distribution}    &\small{{\bf Parameter $\theta$}}    &\small{}             &\small{{\bf Sample Size}}    &\small{}   \\
\small{}      &\small{}   &\small{$n=10$}            &\small{$n=20$}   &\small{$n=30$}\\
\hline
\small{}    &\small{1}           &\small{0.4693}          &\small{0.7078}                   &\small{0.7870}\\
\small{\bf Gamma Family} \small{}    &\small{2}           &\small{0.8635}          &\small{0.9794}                   &\small{0.9914}\\
   &\small{3}           &\small{0.9641}          &\small{0.9980}                   &\small{0.9993}\\
\hline

\small{}    &\small{1}           &\small{0.7672}          &\small{0.9700}                   &\small{0.9920}\\
\small{\bf Linear Failure Rate Family} \small{}    &\small{2}           &\small{0.8455}          &\small{0.9890}                   &\small{0.9977}\\
   &\small{3}           &\small{0.8767}          &\small{0.9929}                   &\small{0.9999}\\
\hline

\small{}    &\small{1}           &\small{0.4641}          &\small{0.7093}                   &\small{0.7754}\\
\small{\bf Weibull Family} \small{}    &\small{2}           &\small{0.9930}          &$\approx$1                  &$\approx$1                  \\
   &\small{3}           &$\approx$1                           &$\approx$1                                    &$\approx$1                  \\
\hline
\end{tabular}
\end{center}
\end{table}

Abouammoh et.al \citep{AAQ} worked on data which represented 40 patients suffering from leukemia from a Ministry of Health Hospital in Saudi Arabia. The life time number of days, in increasing order, are: 115, 181, 255, 418, 441, 461, 516, 739, 743, 789, 807, 865, 924, 983, 1024, 1062,1063, 1165, 1191, 1222, 1222, 1251, 1277,
1290, 1357, 1369, 1408, 1455, 1478, 1549, 1578, 1578,1599, 1603, 1605, 1696, 1735, 1799, 1815, 1852.\\
The test statistic for this data set is ${\hat \Delta}=-0.871605$. When $\alpha=0.05$, ${\hat \sigma}_0^{-1}n^{\frac{1}{2}}{\hat \Delta}=-3.615245 \ll -1.6448$. Thus, from this calculated value, we must reject the null hypothesis and we conclude that this specific set of data has the $ODL$ property.

\section{Conclusion}
The main goal of this paper is to apply nonparametric testing procedures to test if the life time of a system, driven by DMN, falls in the overall decreasing life (ODL) class or not. The mathematical tools and the proposed test introduced here are of general application and should be of general interest to those individuals studying biostatistics, biology, and safety management systems engineering. The results of this paper can also be easily extended to situations involving right-censored data. It is our hope that this method and its application will open the way to the study of more complicated and realistic examples.
\bibliographystyle{authordate4}
\bibliography{ms}
\end{document}